\documentclass[twoside,12pt]{amsart}
\usepackage{amsfonts}
\usepackage{amsmath}
\usepackage{amssymb}
\usepackage{latexsym}
\usepackage{accents}
 \usepackage{indentfirst}
 \usepackage{color}
\makeatletter
\def\serieslogo@{}
\makeatother
\makeatletter
\def\@setcopyright{}
\makeatother

\newtheorem{thm}{Theorem}[section]
\newtheorem{cor}[thm]{Corollary}
\newtheorem{lem}[thm]{Lemma}
\newtheorem{prop}[thm]{Proposition}
\newtheorem{assum}[thm]{Assumption}
\newtheorem*{pos1}{Postulate I}
\newtheorem*{pos2}{Postulate I'}
\theoremstyle{definition}

\theoremstyle{remark}
\newtheorem{rem}{Remark}[section]

\newcommand{\thmref}[1]{Theorem~\ref{#1}}

\newcommand{\lemref}[1]{Lemma~\ref{#1}}

\newcommand{\propref}[1]{Proposition~\ref{#1}}

\newcommand{\bb}[1]{\mathbf{#1}}
\newcommand{\bbtau}{{\boldsymbol \tau}}
\newcommand{\bbzeta}{{\boldsymbol \zeta}}

\newcommand{\e}{\epsilon}

\newcommand{\setstyle}[1]{{\mathbb #1}}

\def\setR {\setstyle{R}}

\begin{document}

\title{A Generalized hydrodynamics and its classical hydrodynamic limit}

\author{Zaibao Yang, Wen-An Yong, Yi Zhu}

\address{Zhou Pei-Yuan Center for Applied Mathematics, Tsinghua University,
Beijing 100084, China}




\begin{abstract}
This work is concerned with our recently developed formalism of non-equilibrium thermodynamics. This formalism extends the classical irreversible thermodynamics which leads to classical thermodynamics and can not describe physical phenomena with long relaxations. With the extended theory, we obtain a generalized hydrodynamic system which can be used to describe the non-Newtonian fluids, heat transfer beyond Fourier's law and fluid flow in very short temporal and spatial scales. We study the mathematical structure of the generalized hydrodynamic system and show that it possesses a nice conservation-dissipation structure and therefore is symmetrizable hyperbolic. Moreover, we rigorously justify that this generalized hydrodynamic system tends to the classical hydrodynamics when the relaxation times approach to zero. This shows that the classical hydrodynamics can be derived as an approximation of our generalized one.
\end{abstract}
\maketitle
\markboth{}{Classical hydrodynamic limit}


%


\section{Introduction}

Consider a one-component fluid system where relativistic and external effects are not taken into account. Under the continuum hypothesis, the system evolves according to the conservation laws of mass, momentum and energy \cite{LaLi}:
\begin{equation}\label{11}
\begin{array}{l}
\partial_t\rho+ \nabla\cdot(\rho \bb v)  = 0, \\[0.5ex]
\partial_t(\rho  \bb v) + \nabla\cdot(\rho \bb v\otimes \bb v + \bb P)=0,\\[0.5ex]
\partial_t(\rho e) + \nabla\cdot(\bb v\rho e+\bb q+\bb P\cdot\bb v)  = 0.
\end{array}
\end{equation}
Here $\rho=\rho(\bb x, t)$ is the fluid density with $(\bb x,t)\in \mathbb{R}^d\times(0,+\infty)$ ($d =1, 2, 3$), $\bb v=\bb v(\bb x, t)\in \setR^d$ is the velocity, $e=e(\bb x, t)$ is the specific energy, $\bb q\in \setR^d$ is the heat flux, and $\bb P\in \mathbb{M}_s^{d}$ is the stress tensor. Hereafter, we denote by $\partial_t$ the partial derivative with respect to the time variable $t$, $\nabla=(\partial x_1,~\partial x_2,\cdots,~\partial x_d)^T$ is the gradient operator with respect to the spatial variable $\bb x$, the dot ``$\cdot$" is the usual tensor contraction, the notation ``$\otimes$" denotes the tensor product, and $\mathbb{M}_s^{d}\cong \mathbb{R}^{d(d+1)/2}$ stands for the $d\times d$ symmetric matrix space with the inner product defined as $\bb A:\bb B=\text{tr}(\bb A^T\bb B)$ where $\text{tr}(\cdot)$ is the trace of a matrix. In \eqref{11}, there are $(d + 2)$ equations for the $(2 + 2d + d(d+1)/2)$ unknowns $\rho, \bb v, e, \bb q$ and $\bb P$. Clearly, the system of equations in \eqref{11} is not closed.


Traditionally, the above system is closed by using Newton's
law of viscosity and Fourier's law of heat conduction. These two empirical laws, together with the conservation laws above,
form the well-known Navier-Stokes-Fourier equations,
which are also referred to as classical hydrodynamics.
On the other hand, it was noted \cite{GM} that the empirical laws can be derived from the non-equilibrium thermodynamics, developed by Onsager, Prigogine and many others, often referred to as classical irreversible thermodynamics (CIT). Indeed, CIT assumes that thermodynamic fluxes depend linearly on thermodynamic forces. For isotropic systems, a direct application of Curie's principle leads to Newton's law of viscosity and Fourier's law of heat conduction.

Even though the classical hydrodynamics can describe a large class of real fluids and has vast applications in engineering and sciences, it is not adequate for describing physical systems such as fluids with memories and heat transfer at high frequencies.
Such situations are met when the relaxation times of the fluxes or stresses are long as in polymer solutions, suspensions \emph{etc}. \cite{EIT_book}. In order to deal with the non-classical situations,
various extended theories of the non-equilibrium thermodynamics have been developed, such as Extended Irreversible Thermodynamics (EIT), Internal Variables Thermodynamics, Rational Thermodynamics, Rational Extended Thermodynamics (RET) and GENERIC \cite{Many_School,EIT_book,NT_2008,Truesdell_book,Muller_book,GO,BET_book}. Each of these theories has its own successes in some aspects, but none of them has been well recognized as the CIT. For instance,
EIT works for the long relaxation phenomena to a certain extent, but it might not
be adequate for systems far away from equilibrium \cite{EIT_book} and the well-posedness of the resultant governing equations does not seem clear;
GENERIC has wide applications in rheology, but it
involves complicated bracket algebras and the resultant governing equations seem not amenable to existing numerics;
and so on. Moreover, these extended theories have not paid much attention to the corresponding
short-relaxation-time limit which is closely related to the compatibility with the CIT \cite{Many_School}.
Consequently, non-equilibrium thermodynamics is not an established edifice, but a work in progress with many different approaches \cite{NT_2008}.


In our previous work \cite{CDF}, we developed a conservation-dissipation formalism (CDF) of non-equilibrium thermodynamics. This theory was inspired by both the EIT \cite{EIT_book} and the structural stability conditions proposed in \cite{Yong92,Yong_book,Yong_ARMA_04,Yong_08} for hyperbolic relaxation systems.
It adopts the advantages of the (at least) three popular schools: EIT, RET and GENERIC.
The structural conditions can be reviewed as stability criteria for non-equilibrium thermodynamics and ensure that non-equilibrium states tend to equilibrium in long time.
The governing equations obtained by this theory have a unified elegant form, are globally hyperbolic and allow a convenient definition of weak solutions.


The purpose of this paper is
to demonstrate the compatibility of the generalized hydrodynamics derived via our CDF with the classical hydrodynamics in the short relaxation-time limit. To do this, we need to introduce proper scalings. Then
we will prove that, as the relaxation time tends to zero, smooth solutions to the generalized hydrodynamics exist in the time interval where the Navier-Stokes-Fourier equations have smooth solutions and
converge to the latter in a proper Sobolev space. Namely, we show that the generalized hydrodynamic system is a hyperbolic approximation to the Navier-Stokes-Fourier equations. Our analysis is guided by the convergence-stability principle \cite{Yong_book} for initial-value problems of symmetrizable hyperbolic systems. It relies crucially on the conservation-dissipation structure and involves a construction of approximate solutions via the Maxwell iteration.


Let us remark that, despite being quite similar in the analysis, the present problem is quite different from those studied in \cite{Yong_ARMA_14,Yong_CPDE,Peng}. Ref. \cite{Yong_ARMA_14} was concerned with an isothermal viscoelastic model with specific linear source terms, while our generalized hydrodynamics is for non-isothermal fluids with general nonlinear source terms. In addition, it is clear that
our diffusive-relaxation problem \eqref{Compact} is not completely included in the classes investigated in \cite{Yong_CPDE,Peng}.
Thus, the present diffusive-relaxation problem requires innovative treatments involving the intrinsic structure of the underlying system.

%


The paper is organized as follow. In Section 2 we give a brief review of the CIT and classical hydrodynamics. The CDF is revisited in Section 3 to obtain the generalized hydrodynamics. In Section 4 we present a formal derivation of the classical hydrodynamics from the generalized one with the short relaxation limit as well as our main theorem. The remaining sections are devoted to a detailed proof of the main theorem: the mathematical structure of the generalized hydrodynamics is discussed in Section 5, the convergency-stability principle is reviewed in Section 6, approximate solutions are constructed in Section 7, and the required error estimates are conducted in Section 8.


\section{Classical Hydrodynamics}
\setcounter{equation}{0}

In this section, we present a brief review of classical irreversible thermodynamics (CIT) and the corresponding hydrodynamics \cite{GM}.
CIT starts with a fundamental assumption called local equilibrium hypothesis. It assumes that the infinitesimal system at $(\bb x, t)$ is always at equilibrium (or infinitely close to equilibrium). Under this hypothesis, all equilibrium thermodynamic concepts \cite{Callen} are valid locally. Thus a non-equilibrium system is completely characterized by equilibrium variables.

For the one-component fluid system considered in this work, only the specific volume $\nu=\frac{1}{\rho}$ and the specific internal energy $u=e-\frac{1}{2}|\bb v|^2$ are needed to describe the infinitesimal system. In addition, the following entropy postulate is generally made:
\begin{pos1}\label{Entropy_Stab}
There exists a differentiable function (called the equilibrium specific entropy) $s^{eq}=s^{eq}(\nu, u)$ satisfying
\begin{enumerate}
\item $s^{eq}(\nu,u)$ is a strictly concave function;
\item $s^{eq}(\nu,u)$ is a strictly monotonically increasing function in $u$, i.e.,  $s^{eq}_u>0$;
\item the production rate of $s^{eq}$ is non-negative.
\end{enumerate}
\end{pos1}
\noindent Correspondingly, the total differential of $s^{eq}$ yields the Gibbs relation
\begin{equation}\label{Gibbs_eq}
\hbox{d}s^{eq}=T^{-1}(\hbox{d}u +p\hbox{d}\nu),
\end{equation}
where $T=\left(\frac{\partial s^{eq}}{\partial u}\right)^{-1}$ and $p=T\frac{\partial s^{eq}}{\partial \nu}$ are the equilibrium temperature and pressure, respectively.

\begin{rem}
The concavity property is also called as the stability criterion of (equilibrium) thermodynamic \cite{Callen}. It follows from the additivity of entropy and the second law of thermodynamics. This criterion guarantees that the heat capacity and isothermal compressibility are positive. Here the arguments $\nu$ and $u$ are required to be the specific values of extensive variables. The monotonicity property is equivalent to that the absolute temperature is positive, i.e., $T>0$. The non-negativity of the entropy production rate reflects the irreversibility of the macroscopic systems.
\end{rem}

With the above postulate and the Gibbs relation \eqref{Gibbs_eq}, the balance equation of the entropy can be easily derived \cite{GM}:
\begin{eqnarray*}
\partial_t(\rho s^{eq})+\nabla\cdot(\rho\bb v s^{eq})+\nabla\cdot(T^{-1}\bb q)=\sigma^{eq},
\end{eqnarray*}
where the entropy production rate is
\begin{equation*}
\sigma^{eq}=\bb q\cdot\nabla T^{-1}-T^{-1}\bbtau:\frac{1}{2}(\nabla \bb v+\nabla \bb v^T).
\end{equation*}
Here $\bbtau=\bb P-p\bb I_d$ is the viscosity stress with $\bb I_d$ the identity matrix of order $d$.

Observe that the production $\sigma^{eq}$ can be rewritten as
\begin{equation*}
\sigma^{eq}=\bb Y^{eq}\cdot \bb X^{eq}=\left(\begin{array}{l}\bb q\\T^{-1}\bbtau\end{array}\right)\cdot \left(\begin{array}{l}\nabla T^{-1}\\-\frac{1}{2}(\nabla \bb v+\nabla \bb v^T)\end{array}\right).
\end{equation*}
Hereafter, for convenience, a matrix in $\mathbb{M}_s^{d}$ is often regarded as to a vector in $\mathbb{R}^{d(d+1)/2}$ when writing in a compact form. The components of the first factor $\bb Y^{eq}=(\bb q,T^{-1}\bbtau)^T$ are fluxes of the conserved variables. They are called as thermodynamic fluxes. The other factor $\bb X^{eq}=(\nabla T^{-1},-\frac{1}{2}(\nabla \bb v+\nabla \bb v^T)^T$ consists of spatial gradients of certain intensive state variables. This factor describes the non-uniformities of the system. In CIT, these non-uniformities are treated as driven forces for systems from non-equilibrium to equilibrium. They are called thermodynamic forces. When the system is at equilibrium, both the thermodynamic forces and the entropy production rate vanish.

Thanks to the causality, the thermodynamic fluxes vanish at equilibrium as well. Thus, we have
the following relation
\begin{equation}\label{22}
\bb Y^{eq}=\bb K\cdot \bb X^{eq},
\end{equation}
where $\bb K$ is called dissipation matrix.
In CIT, only linear regimes are considered. Namely, $\bb K=\bb K(\nu,u)$ does not explicitly depend on the forces, i.e., the spatial derivatives. For an isotropic fluid system, applying Curie's principle \cite{GM}
shows that the dissipation matrix is determined by
\begin{eqnarray}\label{FNS}
\bb K\cdot\left(\begin{array}{c}\nabla T^{-1}\\-\frac{1}{2}(\nabla \bb v+\nabla \bb v^T)\end{array}\right)=\left(\begin{array}{c}-\lambda \nabla T\\ \bb -T^{-1}\bb D[\bb v]\end{array}\right)
\end{eqnarray}
with
\[
\bb D[\bb v]=\xi[\frac{1}{2}(\nabla \bb v+\nabla \bb v^T)-\frac{1}{d}\nabla \cdot \bb v \bb I_d]+\kappa\nabla \cdot \bb v \bb I_d.
\]
Here positive constants $\lambda, \xi$ and $\kappa$ are referred to as heat conduction, shear viscosity, and bulk viscosity coefficients, respectively. We will denote this specific dissipation matrix by $\bb K^{FNS}$
(Fourier-Newton-Stokes matrix).


With constitutive equations \eqref{22} and \eqref{FNS}, the conservation laws \eqref{11} form a closed system of PDEs:
 \begin{equation}\label{NSEq}
\begin{array}{l}
\partial_t\rho+ \nabla\cdot(\rho \bb v)  = 0, \\[0.5ex]
\partial_t(\rho  \bb v) + \nabla\cdot(\rho \bb v\otimes \bb v + pI)-\nabla\cdot \bb D[\bb v]=0,\\[0.5ex]
\partial_t(\rho e) + \nabla\cdot(\bb v\rho e+p\bb v)-\nabla\cdot (\bb D[\bb v]\cdot \bb v+\lambda\nabla T)  = 0.
\end{array}
\end{equation}
Here pressure $p=p(\rho,\rho\bb v, \rho e)$ and temperature $T=T(\rho,\rho\bb v,\rho e)$ are defined by using the Gibbs relation \eqref{Gibbs_eq} (Recall that $u = e - |\bb v|^2/2$). This is the well-known system of Navier-Stokes equations and constitutes the fundamentals of classical hydrodynamics. It describes non-isothermal compressible Newtonian fluid flows and has been widely used in the last two centuries.

Remark that, in CIT, the thermodynamic forces reflect the non-uniformity of the system (for instance, $\nabla T$) and the dissipative fluxes are linearly proportional to the thermodynamic forces (for instance, $\bb q=-\lambda \nabla T$). These linear algebraic constitutive relations guarantee a positive entropy production rate and thereby the second law of thermodynamics.
On the other hand, CIT assumes that the thermodynamic fluxes respond to the thermodynamic forces instantaneously. In other words, the relaxation times of the fluxes induced by the forces are negligible for the evolution of the macroscopic state variables. This instantaneity assumption implies that the scope of application of the classical hydrodynamics is limited. Consequently, extending CIT becomes urgent to study the phenomena with long relaxations.
This explains why there are so many schools of non-equilibrium thermodynamics \cite{Many_School}.

\section{Conservation-Dissipation Formalism}
\setcounter{equation}{0}

In this section, we
introduce our recently developed conservation-dissipation formalism (CDF) of non-equilibrium thermodynamics \cite{CDF} and a generalized hydrodynamics derived via the CDF.

In order to describe the long relaxation phenomena, modern non-equilibrium thermodynamics usually enlarge the state space by adding more state variables \cite{EIT_book}. How to choose suitable extra variables is generally not clear. In CDF, the extra variables are not specified at the beginning. The choice depends on the physical and mathematical structures of the resultant systems. In addition to equilibrium variables $(\nu,u)$, we also need extra non-equilibrium variables denoted by $\bb z \in \mathbb{R}^n$. Thus, the thermodynamic state space is enlarged to an open set $\mathbb{G}\subset\left\{(\nu, u, \bb z)\in \mathbb{R}\times \mathbb{R}\times\mathbb{R}^n: \nu >0\right\}.$

In CDF, the equilibrium thermodynamic postulate (Postulate I) was extended to the enlarged state space.
\begin{pos2}
There exists a smooth function (called the non-equilibrium specific entropy) $s=s(\nu, u, \bb z)$ satisfying
\begin{enumerate}
\item $s(\nu,u,\bb z)$ is a strictly concave function in $\mathbb{G}$;
\item $s(\nu,u,\bb z)$ is a strictly monotonically increasing function in $u$, i.e.,  $s_u>0$;
\item the production rate of $s(\nu, u, \bb z)$ is $s_\bb z\cdot \bb M \cdot s_\bb z$ with
 $\bb M=\bb M(\nu, u,\bb z)$ being positive definite in $\mathbb{G}$.
\end{enumerate}
\end{pos2}

Accordingly, we define conjugate variables as
\[
\theta^{-1}=s_u, \qquad \pi=\theta s_\nu,\quad \bbzeta=s_\bb z.
\]
Thus, we have
\begin{equation}
ds=\theta^{-1}(du+\pi d\nu)+\bbzeta\cdot d\bb z.
\end{equation}
In order to be compatible with the classical theory, $\theta$ is called non-equilibrium temperature and $\pi$ is called non-equilibrium pressure. In addition, $\bbzeta$ is referred to as dissipative entropic variable vanishing at equilibrium.

In Postulate I', the concavity and monotonicity are the same as in Postulate I for equilibrium thermodynamics. The entropy production rate is specified as a quadratic form of the dissipative entropic variable. Namely, the dissipative entropy variable $s_\bb z$ is the thermodynamic force which drives the non-equilibrium system to equilibrium. Obviously, the system is at equilibrium, i.e., the entropy production rate is zero if and only if the driven force is zero, i.e., $s_\bb z=0$ since $\bb M$ is positive definite. In other words, the equilibrium state space is defined as following
$$\mathbb{G}^{eq}=\{(\nu, u, \bb z) \in \mathbb{G}: s_\bb z=0\}.$$

Note that the concavity suggests the extra non-equilibrium variable $\bb z$ behaves like specific values of extensive variables as $\nu$ and $u$. The Reynolds transport theorem indicates that we should seek a governing equation of the form:
\begin{equation}\label{non-eq}
\partial_t(\rho \bb z)+\nabla\cdot\Phi=g,
\end{equation}
where flux $\Phi$ and source $g$ are to be determined. Under Postulate I', CDF proposed a natural way to associate $\bb z$ with the unknown dissipative variables in the conservation laws \eqref{11} and then determines $\Phi$ and $g$ simultaneously. Thus a close systems of evolutionary PDEs is obtained. Next we implement this theory for one-component fluids.

Since the extra unknown variables are $\bb P$ and $\bb q$ for the one-component fluid, the non-equilibrium variables are chosen to be of the same size of the unknown fluxes, i.e., $\bb z=(\bb w, \bb c)\in \mathbb{R}^d\times\mathbb{M}_s^{d}$. In order to specify the flux $\Phi$ and the source $g$ in \eqref{non-eq}, we calculate the balance equation for the entropy to obtain
\begin{align*}
&\partial_t(\rho s)+\nabla\cdot(\bb v \otimes \rho s)\\
=& - \nabla\cdot(\theta^{-1}\bb q)+ s_\bb w\cdot\left[\partial_t(\rho \bb w)+\nabla\cdot(\bb v \otimes \rho \bb w)\right]+\bb q\cdot\nabla \theta^{-1}\\
&+s_\bb c^T:\left[\partial_t(\rho \bb c)+\nabla\cdot(\bb v \otimes \rho \bb c)\right]-\theta^{-1}(\bb P-\pi \bb I)^T:\nabla\bb v.
\end{align*}
With this equation, we refer to \cite{EIT_book} and take $\theta^{-1}\bb q$ as the entropy flux and the rest as the entropy production rate.

Recall that $s=s(\nu, u, \bb w, \bb c)$ is given while $\bb q$ and $\bbtau:=\bb P-\pi \bb I$ are unknown. We choose
$$
\bb q= s_\bb w, \quad \quad ~\bbtau = \theta s_\bb c.
$$
Then the entropy production rate can be written as
$$
\sigma=\bb q\cdot\left[\partial_t(\rho \bb w)+\nabla\cdot(\bb v \otimes \rho \bb w) + \nabla \theta^{-1}\right]\\
+\theta^{-1}\bbtau^T:\left[\partial_t(\rho \bb c)+\nabla\cdot(\bb v \otimes \rho \bb c)-\nabla\bb v\right].
$$
Then a closed system is arisen based on Postulate I' \cite{CDF}. Namely,
\begin{equation}\label{33}
  \left(\begin{array}{c}
\partial_t(\rho\bb w)+\nabla\cdot(\rho\bb v\otimes\bb w)+\nabla\theta^{-1}\\
\partial_t(\rho\bb c)+\nabla\cdot(\rho\bb v\otimes\bb c)-\frac{1}{2}(\nabla \bb v+\nabla \bb v^T)
  \end{array}\right)=\mathbf{M}\cdot  \left(\begin{array}{c}
    \bb q\\
    \theta^{-1}\bbtau
  \end{array}\right)
\end{equation}
with $\mathbf{M}$ positive definite. Consequently, the entropy production rate
$$
\sigma=s_\bb z\cdot \bb M s_\bb z.
$$
It is always positive as long as $s_\bb z$ is not zero. The system reaches to equilibrium when $s_\bb z$ is zero. In this context, the derivative of the entropy with respect to the non-equilibrium variables is regarded as the etropic force which drives the system to equilibrium.

The conservation laws \eqref{11} together with constitutive laws \eqref{33} form a closed system of first-order PDEs. We will see in the later section that it is a symmetrizable hyperbolic system endowed with an entropy. It is also Galilean invariant since $\mathbf{M}$ depends on the velocity $\bb v$ via the local thermodynamic state variable $u$. Further more, it satisfies the Yong stability conditions in \cite{Yong_JDE,Yong_08}. This system is obtained via a non-equilibrium thermodynamic theory beyond local equilibrium hypothesis and is referred to as the generalized hydrodynamic system. If $d=3$, the system is a 13-field equation. We note that to obtain a well-posed 13-field system from Boltzmann equation is also a task of kinetic theory \cite{RuoLi}. However, The traditional moment closure approaches, for example, Grad's approximation, are still struggling for keeping the essential properties of Boltzmann equation like the global hyperoblicity and H-theorem. From another aspect, our macroscopic approximation \eqref{33}  might shed some light on how to construct a mathematically valid moment-closure system.

\section{Main results}
\setcounter{equation}{0}

CDF provides a framework for a thermodynamic description of the macroscopic systems. It has two freedoms--the generalized entropy $s$ and the dissipation matrix $\bb M$. These freedoms are problem-dependent and provide flexibilities in modeling different physical systems. Till now, we only require the concavity and monotonicity of the entropy and positivity of the dissipation matrix. To specify them, other physical considerations are needed. In this section, we consider the compatibility with the classical theory.

When the dissipative entropy variable $\bbzeta$ approaches zero, the system approaches to local equilibrium. In other words, the thermal relations approach to local equilibrium thermal relations and the dissipation feature should be same as that in classical case.  Namely, we have the following compatibility assumption.
\begin{assum}\label{CPT}
At equilibrium manifold $\mathbb{G}^{eq}$, i.e., $s_\bb z=0$, the generalized entropy $s$ and dissipation matrix $\bb M$ are reduced back to the equilibrium forms, i.e.,
$$ s(\nu, u,\bb z)=s^{eq}(\nu, u), \quad \hbox{and} \quad \bb M(\nu,u,\bb z)=(\bb K^{FNS})^{-1}.$$
\end{assum}

Recall that the equilibrium temperature $T$ and pressure $p$ are defined in \eqref{Gibbs_eq}. A direct consequence of this compatibility assumption is that $\theta\to T$ and $\pi\to p$ when the system approaches to equilibrium.


In thermodynamics, the conjugate pairs of thermodynamic forces and fluxes satisfy the so-called thermodynamic causality \cite{GM}. In this work we inherit this causality and assume both the dissipative variables disappear disappear as well when the system reaches to local equilibrium.
\begin{assum}\label{Causal}The dissipative variables vanish at local equilibrium, i.e.,
$$
\bb z=0, \quad \hbox{if} \quad s_\bb z=0.
$$
\end{assum}

A typical choice of such entropy and dissipation matrix is proposed in Appendix and the resulted constitutive equations are also obtained correspondingly.


Up to now, we have not considered the relaxation scales. We assume that the time scales for the relaxations of the dissipative variables $(\bb w, \bb c)$ are $\e_1$ and $\e_2$ respectively. For convenience, we normalized the dissipative variables with their relaxations. Namely, we introduce
$$\tilde{\bb w}=\bb w/\e_1,\qquad \tilde{\bb c}=\bb c/\e_2.$$
Correspondingly, the conjugate variables become
$$\tilde{\bb q}=\e_1\bb q, \quad \text{and} \quad  \tilde{\bbtau}=\e_2 \bbtau.$$
\emph{For the notational convenience, we drop the tildes in the rest of the paper.}

After the rescaling, the generalized hydrodynamic system \eqref{11} and \eqref{33} is of the form
\begin{subequations}\label{GHE}
\begin{eqnarray}
&&\partial_t\rho+ \nabla\cdot(\rho \bb v)  = 0, \\[0.1ex]
&&\partial_t(\rho  \bb v) + \nabla\cdot(\rho \bb v\otimes \bb v +  \pi \bb I_d) +\frac{1}{\e_2}\nabla\cdot\bbtau=0,\\[0.1ex]
&&\partial_t(\rho e) + \nabla\cdot(\bb v\rho e+\pi\bb v)+\frac{1}{\e_1}\nabla\cdot\bb q+\frac{1}{\e_2}\nabla\cdot(\bbtau \cdot\bb v)  = 0,\\[0.1ex]
&&\left(\begin{array}{l}
\partial_t(\rho \bb w) + \nabla\cdot(\rho \bb v\otimes \bb w )+\frac{1}{\e_1}\nabla\theta^{-1}\\
\partial_t(\rho \bb c) + \nabla\cdot(\rho \bb v\otimes \bb c )-\frac{1}{2\e_2}(\nabla \bb v+\nabla \bb v^T)
\end{array}\right)=\left(\begin{array}{ll}\frac{1}{\e_1}&0\\0& \frac{1}{\e_2}\end{array}\right)\mathbf{\overline{M}}\left(\begin{array}{c}
    \frac{1}{\e_1}\bb q\\
     \frac{1}{\e_2}\theta^{-1}\bbtau
  \end{array}\right)\label{wc2}
\end{eqnarray}
\end{subequations}
Here $\overline{\bb M}=\bb M(\nu, u, \e_1\bb w, \e_2\bb c)$.


In many physical processes, the dissipative variables evolve much faster than the macroscopic fluid motions. This is exactly the regime where the classical hydrodynamics is proved to be valid. It is natural to ask  whether the generalized hydrodynamic system \eqref{GHE} has a well-defined limit or this limit is compatible with the classical Fourier-Newton-Stokes equation \eqref{NSEq}.

Rewrite \eqref{wc2} as follows (\emph{for the sake of simplicity, we assume that $\e_1=\e_2=\e$ in this paper}).
\begin{align*}
\left(\begin{array}{c}
\bb q\\
\theta^{-1}\bbtau
  \end{array}\right)=\e \bb M^{-1}(\nu, u, \e\bb w, \e\bb c)\left[\left(\begin{array}{l}\nabla\theta^{-1}\\-\frac{1}{2}(\nabla \bb v+\nabla \bb v^T)\end{array}\right)+\e\left(\begin{array}{l}
\partial_t(\rho \bb w) + \nabla\cdot(\rho \bb v\otimes \bb w )\\
\partial_t(\rho \bb c) + \nabla\cdot(\rho \bb v\otimes \bb c )
\end{array}\right)\right].
\end{align*}

The above formula indicates that  $\left(\bb q,~\theta^{-1}\bbtau\right)=O(\e)$. Then Causality assumption \ref{Causal} implies that $(\bb w, \bb c)=O(\e)$. Using the compatibility of $\bb M$, we obtain that
$$\bb M^{-1}(\nu, u, \e\bb w, \e\bb c)=\bb K^{FNS}+O(\e^2),\quad \theta=T+O(\e^2).$$
Taking an iteration, we immediately get that
\begin{equation}\label{iteration}
\left(\begin{array}{c}
\bb q\\
\theta^{-1}\bbtau
  \end{array}\right)=\e \bb K^{FNS}\left(\begin{array}{l}\nabla T^{-1}\\-\frac{1}{2}(\nabla \bb v+\nabla \bb v^T)\end{array}\right)+O(\e^3)=\e \left(\begin{array}{c}-\lambda \nabla T\\ \bb -T^{-1}\bb D[\bb v]\end{array}\right)+O(\e^3).
\end{equation}
Using $\theta=T+O(\e^2)$ and the definition of $\bb K^{FNS}$, we immediately get that
\begin{eqnarray}
\bb q=-\e\lambda \nabla T+O(\e^3);\\
\bbtau=-\e\bb D[\bb v]+O(\e^3).
\end{eqnarray}
Apparently, the leading terms of the above relations are Fourier's law of heat conduction and Newton-Stokes' law of viscosity.

Notice that $\pi=p+O(\e^2)$ which is directly obtain from Assumptions \ref{CPT} and \ref{Causal}. Substituting \eqref{iteration} into conservation laws in \eqref{GHE}, we have
\begin{equation}
\begin{array}{l}
\partial_t\rho+ \nabla\cdot(\rho \bb v)  = 0, \\[0.5ex]
\partial_t(\rho  \bb v) + \nabla\cdot(\rho \bb v\otimes \bb v + pI)-\nabla\cdot \bb D[\bb v]=O(\e^2),\\[0.5ex]
\partial_t(\rho e) + \nabla\cdot(\bb v\rho e+p\bb v)-\nabla\cdot (\bb D[\bb v]\cdot \bb v+\lambda\nabla T)  = O(\e^2).
\end{array}
\end{equation}
It is seen that the generalized hydrodynamics obtained via CDF behaves the same as the classical hydrodynamics when the relaxation times tend to zero. This compatibility discussion also reveals that the classical Navier-Stokes-Fourier system can be well approximated by a system of first-order PDEs. As a matter of fact, using the idea of CDF, we can find corresponding approximated systems of first-order PDEs for many classical parabolic equations. This lifting from parabolic to hyperbolic is very useful for studying the internal effects like relaxations which parabolic is limited to describe and for designing numerical schemes. It must be pointed out that the above iteration is formal and needs to be rigorously justified.

The main result of this paper is to give a rigorous justification of the above compatibility. Our main theorem is stated as follows.
\begin{thm}\label{thm1}
Under the Assumptions \ref{CPT} and \ref{Causal}, suppose the density $\rho$, velocity $\mathbf{v}$ and energy $e$ of the classical hydrodynamic system \eqref{NSEq} are continuous and bounded in $(x,t)\in\Omega\times[0,t_*]$ with $t_*<\infty$, and satisfy $\displaystyle{\inf_{x,t}\rho(x,t)}>0$ and
$$
\rho, \mathbf{v}, e \in C([0,t_*],H^{s+3})\cap C'([0,t_*],H^{s+1}(\Omega)).
$$
with integer $s\geq [d/2]+2$. Then there exist positive numbers $\e_0=\e_0(t_*)$ and $K=K(t_*)$ such that for $\e\leq \e_0$ the generalized hydrodynamic system \eqref{GHE}, with initial data
in $H^s(\Omega)$ satisfying $\|(\rho^\e,\rho^\e\bb v^\e,\rho^\e e^\e)|_{t=0}-(\rho,\rho\bb v,\rho e)|_{t=0}\|_s=O(\e^2)$, has a unique classical solution satisfying
$$(\rho^\e,\rho^\e\bb v^\e,\rho^\e e^\e,\rho^\e \bb w^\e,\rho^\e\bb c^\e)\in C([0,t_*],H^s(\Omega))$$
and
\begin{equation}
\mathop {\sup }\limits_{t \in [0,{t_*}]}\|(\rho^\e,\rho^\e\bb v^\e,\rho^\e e^\e)-(\rho,\rho\bb v,\rho e)\|_s\leq K(t_*)\e^2.
\end{equation}
\end{thm}Here $\Omega=\mathbb R^d$ or a period domain in $\mathbb{R}^d$, and the notations for the Sobolev spaces are standard as those used in \cite{A.Majda1984}.


In the rest of this paper, we will give a detailed proof of this main result.

The outline of the proof is given as follows.
We first investigate the mathematical structures of the generalized hydrodynamic system \eqref{GHE} in Section \ref{Stru}. These structures are crucial for the proof of our main theorem. It turns out that the generalized hydrodynamic system \eqref{GHE} is a symmetrizable hyperbolic system endowed with a strictly entropy. With the hyperbolicity, this system has a time-local existence for the initial value problem. However the existence time may depend on $\epsilon$ and could tend to zero as $\e\to0$. So we adopt the convergence-stability argument, first formulated in \cite{Yong_book}. In order to apply this argument, we need to construct an approximated solution to the generalized hydrodynamic system \eqref{GHE}. We note that it is more convenient to utilize the so-called normal form. So we rewrite the generalized hydrodynamic system \eqref{GHE} to its normal form. We then construct the approximated solution in Section \ref{Approximated}. With the approximate solution, we then give an alternative compatibility Theorem \ref{CompTheo} based on the normal form. Once \thmref{CompTheo} is proved, our main theorem holds immediately. By applying the convergence-stability argument, we only need to estimate the error between the solution to the normal form and the approximated solution in a compact domain. The detailed estimate is given in Section \ref{Estimate}

\section{\label{Stru}Entropy-Dissipation Structure}
\setcounter{equation}{0}

The purpose of this section is to investigate the mathematical structures of the generalized hydrodynamic system  \eqref{GHE}. This is the mathematical foundation of the main compatibility result.

Let $U^{I}=(\rho, \rho \bb v,\rho e)^T$, $U^{II}=(\rho \bb w, \rho \bb c)^T$ and $U=\left(\begin{array}{c}U^{I}\\ U^{II}\end{array}\right)$. The mathematical state space is
\[
\mathbb{O}:=\left\{(\rho, \rho \bb v, \rho e, \rho \bb w, \rho \bb c): (\frac{1}{\rho}, e-\frac{1}{2}|\bb v|^2, \bb w, \bb c) \in \mathbb{G} \right\}.
\]

We define the mathematical entropy function
\begin{equation}\label{51}
 \eta =\eta(U):=-\rho s(\nu, u, \bb w, \bb c)
\end{equation}
for $U\in \mathbb{O}$. Then the following properties of $\eta(U)$ can be obtained directly (see Appendix).
\begin{prop}\label{prop51}
$\eta(U)$ is a strictly convex function in $\mathbb{O}$ with
\begin{equation}\label{etau}
\eta_U=(\eta_\rho, \theta^{-1}\bb v, -\theta^{-1},-\bb q, -\theta^{-1}\bbtau)^T
\end{equation}
and
\begin{equation}\label{pi}
\theta^{-1}\pi=\eta_U\cdot U-\eta.
\end{equation}
\end{prop}
Hereafter, the thermodynamic functions, without introducing any confusions, are functions of $U$ taking the vales in $\mathbb{O}$, i.e., $\theta=\theta(U)$, $\pi=\pi(U)$, $\bb q=\bb q(U)$ and $\bbtau=\bbtau(U)$.

In this context, equilibrium states are reached when $\eta_{U^{II}}=0$. As the system approaches to equilibrium, the compatibility assumption \ref{CPT} and casualty assumption \ref{Causal} lead to the following estimates.
\begin{prop}\label{prop_comp}
Under Assumptions \ref{CPT} and \ref{Causal}, if $\eta_{U^{II}}\to 0$, the following estimates hold
\begin{enumerate}
\item $U^{II}=O(|\eta_{U^{II}}|)$;
\item $\eta(U)=\eta(U^I,0)+O(|\eta_{U^{II}}|)$;
\item $M(U; \e)=(\bb K^{FNS})^{-1}+O(\e|\eta_{U^{II}}|)$.
\end{enumerate}
\end{prop}

Using the thermodynamic relation \eqref{etau}, we immediate get the following estimate from \propref{prop_comp}.
\begin{cor}\label{cor_comp}
Under Conditions in Proposition \ref{prop_comp}, $\eta_{U^{II}U^{I}}=O(|\eta_{U^{II}}|).$
Thus the thermodynamic functions have the estimates
 $$\pi=p+O(|\eta_{U^{II}}|^2), \quad \theta=T+O(|\eta_{U^{II}}|^2).$$
\end{cor}

Next, we write the generalized hydrodynamic system \eqref{GHE} in the following form,
\begin{equation}\label{Compact}
\partial_tU+ \nabla\cdot F(U)+\dfrac{1}{\e}\nabla \cdot  G(U)=-\dfrac{1}{\e^2}\left( \begin{array}{c}
 0 \\
M^\e(U){\eta _{{U^{II}}}} \\
 \end{array} \right),
\end{equation}
where $M^\e(U)=\bb M(\nu,u,\e \bb w,\e \bb c)$ is the dissipation matrix and
\begin{align}
&\nabla\cdot F(U)=\nabla\cdot \bb v\otimes U+\left(\begin{array}{c}
    0\\
    \nabla\pi\\
    \nabla\cdot\pi \bb v\\
    0\\
    0
    \end{array}\right),\label{funF}\\
&\nabla\cdot G(U)=\left(\begin{array}{c}
    0\\
    \nabla\cdot\bbtau\\
    \nabla\cdot(\bb q+\bbtau \bb v)\\
    \nabla\theta^{-1}\\
    - \frac{1}{2}(\nabla \bb v+\nabla \bb v^T)\end{array}\right)=\left(\begin{array}{c}
    0\\
    -\nabla\cdot\theta\eta_{\rho \bb c}\\
    \nabla\cdot(-\eta_{\rho \bb w}-\theta\eta_{\rho \bb c}\bb v)\\
    \nabla \theta^{-1}\\
    - \frac{1}{2}(\nabla \bb v+\nabla \bb v^T) \label{funG}
    \end{array}\right).
\end{align}
Let $F_j(U)$ and $G_j(U)$ denote the $x_j$-components of the above fluxes, i.e.,  $F(U)=(F_1,F_2,\cdots, F_d)$ and $G(U)=(G_1,G_2,\cdots, G_d)$.

We are now in a position to state the symmetrizable structure of the system.
\begin{prop}\label{Symmetrizable}$\eta_{UU}\cdot F_{jU}$ and $\eta_{UU}\cdot G_{jU}$ are symmetric for $U\in \mathbb{O}$ and $j=1,2,\cdots d$.
\end{prop}
\begin{proof}
A direct calculation shows that
\begin{align*}
\eta_U\cdot (\nabla \cdot F(U))&=\eta_U\cdot(\nabla\cdot(\bb v\otimes \eta))+\theta^{-1}\bb v\cdot\nabla\pi-\theta^{-1}(\nabla\cdot(\pi \bb v))\\
&=\nabla\cdot(\bb v\eta)+(-\eta+\eta_U\cdot U)\nabla\cdot\bb v-\theta^{-1}\pi\nabla\cdot\bb v\\
&=\nabla\cdot(\bb v\eta).
\end{align*}
where we have used \eqref{pi} to get the last equality. Since the above relation holds for arbitrary $\partial_{x_j}U$, we have
\begin{align*}
\eta_{U}\cdot  F_{jU}=(\bb v_j \eta)_U.
\end{align*}
 Taking derivative of both sides of the above equations with respect to $U$, we have
\begin{align}
\eta_{UU}\cdot  F_{jU}+\eta_{U}\cdot  F_{jUU}=(\bb v_j \eta)_{UU}.
\end{align}
Using the symmetry of the matrices $\eta_{U}\cdot  F_{jUU}$ ~and~ $(\bb v_j \eta)_{UU}$, we find that $\eta_{UU} F_{jU}$ is symmetric.

On the other hand, using \eqref{funG} and $\eta_{\rho \bb v}=\theta\bb v$, we immediately get
\begin{align*}
\eta_U\cdot (\nabla \cdot G(U))&=\theta^{-1}\bb v\cdot (-\nabla\cdot\theta \eta_{\rho \bb c})-\theta^{-1}\nabla\cdot(-\eta_{\rho \bb w}-\theta\eta_{\rho \bb c}\bb v)+\eta_{\rho \bb w}\cdot\nabla\theta^{-1}-\frac{1}{2}\eta_{\rho \bb c}:(\nabla \bb v+\nabla \bb v^T)\\
&=\nabla\cdot(\theta^{-1}\eta_{\rho \bb w})+\theta^{-1}\left(\nabla\cdot(\theta\eta_{\rho\bb c}\bb v)-\bb v\cdot \nabla\cdot\theta \eta_{\rho \bb c}-\eta_{\rho \bb c}:\nabla \bb v\right)\\
&=\nabla\cdot(\theta^{-1}\eta_{\rho \bb w}).
\end{align*}
where we have used that $\eta_{\rho \bb c}$ is symmetric and thus $\frac{1}{2}\eta_{\rho \bb c}:(\nabla \bb v+\nabla \bb v^T)=\eta_{\rho \bb c}:\nabla \bb v$.
A similar argument leads to
\begin{align}
\eta_{U}\cdot G_{jU}=(\theta^{-1}\eta_{\rho \bb w_j})_U.
\end{align}
and the conclusion that $\eta_{UU} G_{jU}$ is symmetric.

\end{proof}

With the above preparations, the evolution of the entropy $\eta(U)$ is readily calculated and we reach to the following theorem.
\begin{prop}
The system \eqref{Compact} endowed with the strictly convex entropy $\eta(U)$ in the state space $\mathbb{O}$ is symmetrizable hyperbolic. And the governing equation of the mathematical entropy $\eta(U)$ satisfies
\begin{align}
\partial_t\eta+\nabla\cdot (\bb v \eta)+\frac{1}{\e}\nabla\cdot (\theta^{-1}\eta_{\rho \bb w})=\sigma
\end{align}
where the entropy production
\begin{equation}
\sigma=-\frac{1}{\e^2}\eta_{U^{II}}\cdot M^\e\cdot \eta_{U^{II}}\le0.
\end{equation}
\end{prop}

Next we introduce the normal form \cite{KY1} of the generalized hydrodynamic system \eqref{Compact}. Set
$$
V=\left( \begin{array}{l}
 V^I \\
 V^{II}
 \end{array} \right)=\left( \begin{array}{l}
 U^I \\
 \eta_{U^{II}}(U)
 \end{array} \right).
$$
Notice that this transform has a global inverse $U=U(V)$ thanks to the strict convexity of $\eta(U)$. The governing equation for $V$ is
\begin{equation}\label{71}
\partial_t V+\sum\limits_{j=1}^dA_j(V)\partial_{x_j}V+\frac{1}{\e}\sum\limits_{j=1}^dB_j(V)\partial_{x_j}V=-\dfrac{1}{\e^2}\left( \begin{array}{c}
0 \\
H(V; \e) V^{II} \\
 \end{array} \right),
\end{equation}
where $A_j(V)= J  F_{jU} J^{-1}$ with $J = \frac{\partial V}{\partial U}$, $B_j(V)= J G_{jU} J^{-1}$ and $H(V; \e)=\eta_{U^{II}U^{II}} M(U; \e)$. This is the normal form of the generalized hydrodynamics \eqref{Compact}. It is a symmetrizable hyperbolic system with symmetrizer
$$
A_0(V)= (J^{-1})^T\eta_{UU} J^{-1},
$$
for $A_0(V)$ is symmetric and positive definite, $A_0(V)A_j(V)$ and $A_0(V)B_j(V)$ are symmetric. Moreover, the symmetrizer is block-diagonal \cite{KY1}
\begin{equation}\label{72}
A_0(V)=\left( {\begin{array}{*{20}{c}}
   {\eta_{U^{I}U^{I}}-\eta_{U^{I}U^{II}}\eta_{U^{II}U^{II}}^{-1}\eta_{U^{II}U^{I}}} & {0}  \\
   {0} & {\eta_{U^{II}U^{II}}^{-1}}
\end{array}} \right)\equiv \left({\begin{array}{*{20}{c}}
   {A_0^{I,I}}(V) & 0  \\
   0 & {A_0^{II,II}}(V)
\end{array}} \right).
\end{equation}

$A_0^{II,II}(V)H(V; \e)=M(V; \e)$ is positive definite. In this normal form, the equilibrium state reaches when $V^{II}=0$.



\begin{prop}\label{prop e}
Assume. Then we have
\begin{enumerate}
\item $A_0^{I,I}(V)=\eta_{U^IU^I}(U(V^I,0))+O(|V^{II}|^2)$;

\item $B_j^{I,I}(V^I, {V^{II}=0})=0$.
%
\end{enumerate}
\end{prop}
\begin{proof}
By direct computations, we have
$$
A_0^{I,I}(V)={\eta_{U^{I}U^{I}}-\eta_{U^{I}U^{II}}\eta_{U^{II}U^{II}}^{-1}\eta_{U^{II}U^{I}}}
$$
and
$$
B_j^{I,I}(V)={G^{I}_{j{U^{I}}}} - {G^{I}_{j{U^{II}}}}\eta _{{U^{II}}{U^{II}}}^{ - 1}{\eta _{{U^{II}}{U^{I}}}}.
$$
Since $\eta_{U^{II}}(U), G_j^I(U)$ and $U^{II}$ all vanish at equilibrium, we immediately obtain (1) and (2).
\end{proof}

%
%

\section{\label{Existence} A convergence-stability principle}
\setcounter{equation}{0}

In this section, we introduce the convergence-stability principle \cite{Yong_book} to present a framework for proving the main result \thmref{thm1}. Recall in the previous section that the generalized hydrodynamic system \eqref{Compact} has been shown to be symmetrizable hyperbolic.

Fix $\e$. Let $s>d/2 + 1$ be an integer, $U_0=U_0(x,\e)\in H^s(\Omega)$ and take values in a compact set $G_0\subset\subset G$ for all $(x,\e)$. According to the local existence theory \cite{A.Majda1984} for symmetrizable hyperbolic systems (see Theorem 2.1 in \cite{A.Majda1984}), there is a time interval $[0, \tilde t]$ so that the generalized hydrodynamic system \eqref{Compact} has a unique $H^s$-solution $U^\e\in C([0,\tilde t],H^s(\Omega))$ with initial value $U_0(x,\e)$. Notice that the time interval for the existence depends on $\e$ in general. For $G_1\subset\subset G$ satisfying $G_0\subset\subset G_1$, we define
\begin{equation}\label{61}
t_\e =\sup\{\tilde t>0: U^\e(\cdot, t)\in H^s(\Omega), U^\e(x,t)\in G_1\}.
\end{equation}
Namely, $[0,t_\e)$ is the maximal time interval of $H^s(\Omega)$-existence. Note that $t_\e$ may tend to 0 as $\e$ goes to 0.

In order to show that $\liminf_{\e\rightarrow0} t_\e>0$, we adopt the convergence-stability principle formulated in \cite{Yong_book}. Namely, we suppose that an approximate solution $U_\e=U_\e(x, t)$ has been found and satisfies the following

\noindent {\it{Convergence Assumption:}}  There exists $t_*>0$ so that $U_\e\in L^\infty([0,t_*],H^s)$ for each $\e>0$ and possesses the following properties
$$
\mathop{\cup}\limits_{x,t,\e}\{U_\e(x,t)\}\subset\subset G_1
$$
and for $t\in[0,\min\{t_*,t_\e\})$,
$$
\sup\limits_{x,t}|U^\e(x,t)-U_\e(x,t)|=o(1),
$$
$$
\sup\limits_{t}\parallel U^\e(\cdot,t)-U_\e(\cdot,t)\parallel_s=O(1)
$$
as $\e$ tends to 0.

With $U_\e=U_\e(x, t)$ satisfying the above assumption, the following fact was established in \cite{Yong_book}:
\begin{lem}\label{lem61}
Let $s>d/2 + 1$ be an integer, $U_0=U_0(x,\e)\in H^s(\Omega)$ and take values in a compact set $G_0\subset\subset G$ for all $(x,\e)$. Suppose an approximate solution $U_\e=U_\e(x, t)$ has been found and satisfies the convergence assumption above. Then there exists $\e_0>0$ such that, for $\e<\e_0$, the maximal existence time $t_\e$ defined in \eqref{61} is larger than $t_*$ in the convergence assumption: $t_\e>t_*. $
\end{lem}

Thanks to this lemma, our task is reduced to construct the approximate solution $U_\e$ and to estimate $\|U^\e(t)-U_\e(t)\|_s$ for $t\in [0,\min\{t_\e,t_*\})$. In this time interval, both $U^{\e}$ and $U_\e$ are regular enough and take values in the compact set $G_1$.

We conclude this section with the well-known calculus inequalities in Sobolev spaces, which can be found in \cite{A.Majda1984} and will be used for the estimate.
\begin{lem}\label{lem62}
Let $s,s_1$ and $s_2$ be three non-negative integers and $s_0=[d/2]+1$.\\
1). If $s_3=\min\{s_1,s_2,s_1+s_2-s_0\}\geq0$, then $H^{s_1}H^{s_2}\subset H^{s_3}$. Here the inclusion symbol $\subset$ implies the continuity of the embedding.\\
2). Suppose $s\geq s_0+1, A(V)\in H^s$, and $Q(V)\in H^{s-1}$. Then for all
multi-indices $\alpha$ with $|\alpha|\leq s,
[A,\partial_\alpha]Q\equiv A\partial_x^\alpha Q-\partial_x^\alpha(AQ)\in L^2$ and
$$
\| A\partial_x^\alpha Q-\partial_x^\alpha (AQ) \|
\leq C_s\| A\|_s\| Q\|_{|\alpha|-1}.
$$
3). Suppose $s\geq s_0$, $A\in C_b^s(G)$, and $V\in H^s(\Omega, G)$. Then $A(V(\cdot))\in H^s$ and
$$
\| A(V(\cdot))\|_s\leq C_s|A|_s(1+\| V\|_s^s)
$$
Here and below, $C_s$ denotes a generic constant depending only on $s,n$ and $d$, and $|A|_s$ stands for $\sup\limits_{V\in G,|\alpha|\leq s}|\partial_V^\alpha A(V)|$.\\
\end{lem}

\section{\label{Approximated} Approximate solutions}
\setcounter{equation}{0}

The main purpose of this section is to construct an approximate solution $U_\e$ based on the classical hydrodynamic equations \eqref{NSEq}. Let $V^{I}_{\e}=(\rho,\rho \mathbf {v},\rho e)^T$ be a solution to the classical hydrodynamic equations \eqref{NSEq} with $T$ the corresponding equilibrium temperature. We define
\begin{equation}
V^{II}_{\e}=-\e\left(\begin{array}{c}\lambda \nabla T \\ T^{-1}\bb D[\bb v]
\end{array}\right),
\qquad
V_{\e}=\left(\begin{array}{c} V^{I}_{\e}\\ V^{II}_{\e}
\end{array}\right)
\end{equation}
and
\begin{equation}\label{R}
R(V_\e)=\partial_t V_\e+\sum\limits_{j=1}^dA_j(V_\e)\partial_{x_j}V_\e+\frac{1}{\e}\sum\limits_{j=1}^dB_j(V_\e)\partial_{x_j}V_\e+\dfrac{1}{\e^2}\left( \begin{array}{c}
 0\\
H(V_\e; \e)V_\e^{II} \\
 \end{array} \right).
\end{equation}

About this $R(V_\e)\equiv\left(\begin{array}{c}R^I(V_\e)\\ R^{II}(V_\e)\end{array}\right)$, we have
\begin{lem}\label{lem4}
Let $s> d/2+1$ be an integer. Assume that the solution $(\rho,\rho\mathbf{v},\rho e)$ to the classical hydrodynamic equations \eqref{NSEq}
satisfies $(\rho,\rho\mathbf{v},\rho e)\in C([0,t_*],H^{s+2})\cap C^1([0,t_*],H^{s+1})$ with the equilibrium temperature $T$ having a positive lower bound. Then we have $V_\e\in C([0,t_*],H^{s+1})\cap C^1([0,t_*],H^{s})$ and $R(V_\e)\in C([0,t_*],H^s)$. Moreover, we have the following estimates
$$
\|V^{II}_{\e}\|_{s+1}\leq C\e,
\qquad \|R^{I}(V_\e)\|_s\leq C\e^2,\qquad \|R^{II}(V_\e)\|_s\leq C\e.
$$
Here and below, $C$ is a generic constant which may change from relations to relations.  \\
\end{lem}

\begin{proof}
By using the calculus inequalities in Lemma \ref{lem62}, we deduce the regularity of $V_\e$ and $R(V_\e)$ and the estimate of $V^{II}_\e$ directly from the definitions and the conditions of the lemma.

To obtain the estimates of $R^I(V_\e)$ and $R^{II}(V_\e)$, we define $U_\e\equiv U(V_\e)=\eta_U^{-1}(V_\e)$ and recall
$$
\left(\begin{array}{c}\bb q(U_\e) \\ \theta(U_\e)^{-1}\bbtau(U_\e)
\end{array}\right)=\eta_{U^{II}}(U_\e)=V^{II}_\e=-\e\left(\begin{array}{c}\lambda \nabla T \\ T^{-1}\bb D[\bb v]
\end{array}\right).
$$
Moreover, it follows from Corollary \ref{cor_comp} and $U_\e^I = V_\e^I =(\rho,\rho \mathbf {v},\rho e)^T$ that
$$
\pi(U_\e) =\pi(U_\e^I, 0) +O(|V_\e^{II}|^2)=\pi((\rho,\rho \mathbf {v},\rho e), 0) +O(|V_\e^{II}|^2)=p+O(|V_\e^{II}|^2)
$$
and
$$
\theta(U_\e) =T+O(|V^{II}_\e|^2).
$$
Then from the Navier-Stokes equations and the definition of $R^I(V_\e)$ we get that
\begin{equation*}
\begin{split}
 R^{I}(V_\varepsilon)=\nabla  \cdot\left( \begin{array}{c}
 0 \\
 (\pi(U_\e)  - p)\bb I_d + (1-T^{-1}\theta(U_\e))\bb D[\bb v]\\
 (\pi(U_\e)  - p)\bb v + (1-T^{-1}\theta(U_\e))\bb D[\bb v]\bb v \\
 \end{array} \right).
 \end{split}
 \end{equation*}
This simply implies that $\|R^I(V_\e)\|_s\leq C\|V^{II}_\e\|_{s+1}^2\leq C\e^2$.

It remains to estimate $\|R^{II}(V_\e)\|_s$. From the definition \eqref{R}, it is easy to see that
\begin{equation*}
R^{II}(V_\e)=
\eta_{U^{II}U^{II}}\left(\partial_tU_\e^{II}+\nabla \cdot( F^{II}(U_\e)+\frac{1}{\e} G^{II}(U_\e))+\dfrac{1}{\e^2}M(U_\e; \e)V_\e^{II}\right)+\eta_{U^{II}U^{I}}R^{I}(V_\e).
\end{equation*}
Because of Eq. \eqref{funF} and Proposition \ref{prop_comp}, we have $F^{II}(U)=\bb v\otimes U^{II}$ and  $U^{II}(V)=O(|V^{II}|)$, which immediately gives
$$
\|\partial_tU_\e^{II}+\nabla \cdot F^{II}(U_\e)\|_s\leq C\e.
$$
On the other hand, notice that $V^{II}_\e$ is defined so that
$$
\nabla\cdot G^{II}(U(V_\e^I,0))+\dfrac{1}{\e} (\bb K^{FNS})^{-1}V_\e^{II}=0.
$$
Then we have
\begin{eqnarray*}
&&\| \nabla\cdot G^{II}(U_\e)+\dfrac{1}{\e}M(U_\e; \e)V_\e^{II}\|_s\\
=&&\|\nabla\cdot G^{II}(U_\e)+\dfrac{1}{\e}M(U_\e; \e)V_\e^{II}-\nabla\cdot G^{II}(U(V_\e^I,0))-\dfrac{1}{\e} (\bb K^{FNS})^{-1}V_\e^{II}\|_s\\
\le&&\|\nabla\cdot [G^{II}(U_\e)-G^{II}(U(V_\e^I,0))]\|_s +\dfrac{1}{\e}\|[M(U_\e; \e)-(\bb K^{FNS})^{-1}]V_\e^{II}\|_s\\
=&&\|\nabla[\theta(U_\e)^{-1}-T^{-1}]\|_s +\dfrac{1}{\e}\|O(\e|V_\e^{II}|)V_\e^{II}\|_s\\
\le&& C\e^2.
\end{eqnarray*}
Here we have used \eqref{funG} for $G^{II}(U)$, Proposition \ref{prop_comp} for $M(U_\e; \e)=(\bb K^{FNS})^{-1}+O(|\e V^{II}|)$, and the above proved fact $\theta(U_\e)=T+O(|V^{II}|^2)$.
Summing up the above estimates, we obtain
$$
\|R^{II}(V_\e)\|_s\leq C\e.
$$
This completes the proof.
\end{proof}

Denote by $V^\e$ the solution to the normal form \eqref{71} of the generalized hydrodynamic system \eqref{Compact}. In the next section we will prove the following theorem, which is more general than our main result Theorem \ref{thm1}.
\begin{thm}\label{CompTheo}
Let $s>d/2+1$ be an integer. Suppose the Navier-Stokes equations \eqref{NSEq} have a solution
$$
(\rho, \mathbf{v}, e)\in C([0,t_*],H^{s+2}(\Omega))\cap C^1([0,t_*],H^{s+1}(\Omega))
$$
with $t_*<\infty$ and the corresponding density and temperature have positive lower bounds.
Then there exist positive numbers $\e_0=\e_0(t_*)$ and $K=K(t_*)$ such that, for $\e\leq \e_0$, the normal form \eqref{71} with initial data $V_\e(x,0)$ has a unique classical solution $V^\e$ satisfying $$V^\e(x,t)\in C([0,t_*],H^s(\Omega))$$
and
$$
\mathop {\sup }\limits_{t \in [0,{t_*}]}\|V^\e(t)-V_\e(t)\|_s\leq K(t_*)\e^2.
$$
\end{thm}

%
%
%


\section{\label{Estimate}Error estimates}
\setcounter{equation}{0}

We prove Theorem \ref{CompTheo} in this section. Thanks to Lemma \ref{61}, we only need to estimate $\|V^\e(t)-V_\e(t)\|_s$ for $t\in [0,\min\{t_\e,t_*\})$. In this time interval, both $V^{\e}$ and $V_\e$ are regular enough and take values in a compact set. In what follows, we set
$$
E=V^\e-V_\e\equiv \left(\begin{array}{c}
E^I\\
E^{II}
\end{array}\right).
$$

Recall that $V^\e$ solves the equivalent equations in \eqref{71}, while $V_\e$ satisfies the same equations with a residual $R(V_\e)$ \eqref{R}. It follows that the error $E$ satisfies
\begin{equation*}
\begin{split}
\partial_tE+\sum\limits_{j=1}^dA_j(V^\e)\partial_{x_j}E+&\dfrac{1}{\e}\sum\limits_{j=1}^dB_j(V^\e)\partial_{x_j}E=-\dfrac{1}{\e^2}\left( \begin{array}{c}
0 \\
H^\e(V^\e)E^{II} \\
\end{array} \right)-\left( \begin{array}{c}
 R^I \\
 R^{II} \\
 \end{array} \right)\\
 &+\sum_{j=1}^d\big{(}A_j( V_{\e})-A_j( V^{\e})\big{)} \partial_{x_j}V_{\e}+\dfrac{1}{\e}\sum_{j=1}^d\big{(}B_j(V_{\e})-B_j(V^{\e})\big{)}\partial_{x_j}V_{\e }\\
 &+\dfrac{1}{\e^2}\left( \begin{array}{c}
0 \\
\big{(}H^\e(V_\e)-H^\e(V^\e)\big{)}V^{II}_{\e} \\
\end{array} \right).
 \end{split}
\end{equation*}
Differentiating the two sides of the last equation with $\partial_x^{\alpha}, \;|\alpha|\leq s$, and denoting $E_\alpha=\partial_x^{\alpha} E$, we get
\begin{equation*}
\partial_tE_{\alpha}+\sum\limits_{j=1}^d A_j(V^{\e}) \partial_{x_j}E_{\alpha}
+\dfrac{1}{\e}\sum\limits_{j=1}^dB_j(V^{\e}) \partial_{x_j}E_{\alpha }=F_1^\alpha+F_2^\alpha+F_3^\alpha+F_4^\alpha ,
\end{equation*}
where
\begin{equation}\label{81}
\begin{array}{l}
F_1^\alpha=-\dfrac{1}{\e^2}\left( \begin{array}{c}
0 \\
H^\e(V^\e) E^{II}_{\alpha} \\
\end{array} \right)-\left( \begin{array}{c}
 \partial_x^{\alpha}R^{I} \\
 \partial_x^{\alpha}R^{II}
 \end{array} \right),\\
 F_2^\alpha=\sum\limits_{j=1}^d\partial_x^\alpha\Big{\{}\Big{(}A_j(V_{\e})-A_j(V^{\e})+\dfrac{1}{\e}\big{(}B_j(V_{\e})-B_j(V^{\e})\big{)}\Big{)}\partial_{x_j}V_{\e}\Big{\}},\\
 F_3^\alpha=\sum\limits_{j=1}^d\Big{(}[A_j(V^{\e}),\partial_x^{\alpha}] \partial_{x_j}E+\dfrac{1}{\e}[B_j(V^{\e}),\partial_x^\alpha] \partial_{x_j}E\Big{)},\\
 F_4^\alpha=\dfrac{1}{\e^2}\left( \begin{array}{c}
0 \\ {}[H^\e(V^\e), \partial_x^\alpha]E^{II} +\partial_x^\alpha\Big{(}\big{(}H^\e(V_\e)-H^\e(V^\e)\big{)} V^{II}_\e\Big{)}
\end{array} \right).
\end{array}
\end{equation}
Multiplying $E_\alpha^TA_0(V^\e)$ on the above equation and using the fact that $A_0, A_0A_j$ and $A_0B_j$ are all symmetric yield
\begin{equation*}
\begin{split}
&\partial_t \big{(}E_{\alpha}^TA_0(V^\e)E_{\alpha}\big{)} +\sum\limits_{j=1}^d \partial_{x_j}\big{(}E_{\alpha}^TA_0( V^{\e})A_j(V^{\e}) E_{\alpha}\big{)}
+\dfrac{1}{\e}\sum\limits_{j=1}^d\partial_{x_j}\big{(}E_{\alpha}^TA_0(V^{\e})B_j(V^{\e})E_{\alpha}\big{)}\\
=&2E_{\alpha}^TA_0(V^\e)(F_{1}^{\alpha}+F_{2}^\alpha+F_{3}^{\alpha}+F_{4}^{\alpha})\\
&+E_{\alpha}^T\Big{(}\partial_t A_0(V^{\e})+\sum\limits_{j=1}^d \partial_{x_j}\big{(}A_0(V^\e)A_j(V^{\e})\big{)}+\dfrac{1}{\e}\sum\limits_{j=1}^d \partial_{x_j}\big{(}A_0(V^\e)B_j(V^{\e})\big{)}\Big{)}E_{\alpha}.
\end{split}
\end{equation*}
Integrating the last equation over $\Omega$ gives
\begin{equation}\label{82}
\begin{split}
& \dfrac{d}{d t} \int_\Omega E_{\alpha}^TA_0(V^\e) E_{\alpha}dx \\
= &
2\int_\Omega E_{\alpha}^TA_0( V^{\e})( F_1^\alpha+ F_2^\alpha+ F_3^\alpha+ F_4^\alpha)dx +\int_\Omega E_{\alpha}^T\partial_tA_0( V^{\e}) E_{\alpha}dx\\
& +\sum\limits_{j=1}^d \int_\Omega E_{\alpha}^T\partial_{x_j}\Big{(}A_0( V^{\e})A_j( V^{\e})+\dfrac{1}{\e}A_0( V^{\e})B_j( V^{\e})\Big{)} E_{\alpha}dx.
\end{split}
\end{equation}

Next we estimate the right-hand side of \eqref{82} term by term. Recall that $A_0$ is block-diagonal and $A_0^{II,II}(V)H(V; \e)=M(U(V); \e)$ is positive definite. Then we have
\begin{equation}\label{83}
\begin{split}
& \int_\Omega E_{\alpha}^TA_0( V^{\e})F_1^\alpha dx\\
=&-\dfrac{1}{\e^2}\int_\Omega (E_\alpha^{II})^TM(U(V^{\e}); \e)E_\alpha^{II}dx-\int_\Omega \big{(}(E_\alpha^{I})^T \partial_x^\alpha R^{I}+(E_\alpha^{II})^T \partial_x^\alpha R^{II}\big{)} dx\\
\leq & -c_0\dfrac{\|E_\alpha^{II}\|^2}{\e^2}+C\|E_\alpha\|^2+C\e^4
\end{split}
\end{equation}
with $c_0$ a generic positive constant due to the positivity of the matrix $M(U(V^{\e}); \e)$ for $V^{\e}$ taking values in the compact set, where we have used the estimates $\|R^{I}(V_\e)\|_{s}=O(\e^2)$ and $\|R^{II}(V_\e)\|_{s}=O(\e)$ in \lemref{lem4}.

To treat the other terms, we will repeatedly use the Sobolev calculus inequalities in Lemma \ref{lem62}, the boundedness of $\|\nabla V_\epsilon\|_s$ and $\|V^{II}_\e\|_{s+1}=O(\e)$.
We will also use the facts that $V^{\e}$ and $V_{\e}$ take values in the compact set, $A_0(V)$ is block-diagonal \eqref{72}, and $B_j^{I,I}(V)$ vanishes at equilibrium (Proposition \ref{prop e}). From the last fact, it is not difficult to deduce that
$$
\|B_j^{I,I}(V^\e)-B_j^{I,I}(V_\e)\|_{|\alpha|}\leq C(1+\|E\|_s^s)\Big{(}\|E^{II}\|_{|\alpha|}+\|V^{II}_\epsilon\|_s\|E^I\|_{|\alpha|}\Big{)}.
$$
Thanks to the last inequality, we have
\begin{equation}\label{84}
\begin{split}
& \int_\Omega E_{\alpha}^TA_0( V^{\e}) F_2^\alpha dx\\
=& \sum\limits_{j=1}^d\int_\Omega E_{\alpha}^TA_0( V^{\e}) \partial^{\alpha}_{x}\Big{\{}\Big{(}A_j(V_{\e})-A_j(V^{\e})+\dfrac{1}{\e}\big{(}B_j(V_{\e})-B_j(V^{\e})\big{)}\Big{)}\partial_{x_j}V_{\e}\Big{\}}dx\\
\leq& C\|E_\alpha\|\sum\limits_{j=1}^d\|A_j(V^\e)-A_j(V_\e) \|_{|\alpha|}\\
&+C\|E_\alpha^I\|\sum\limits_{j=1}^d\Big{(}\dfrac{1}{\e}\|B_j^{I,I}(V^\e)-B_j^{I,I}(V_\e)\|_{|\alpha|}+\|B_j^{I,II}(V^\e)-B_j^{I,II}(V_\e)\|_{|\alpha|}\Big{)}\\
&+C\|E_\alpha^{II}\|\sum\limits_{j=1}^d\Big{(}\dfrac{1}{\e}\|B_j^{II,I}(V^\e)-B_j^{II,I}(V_\e)\|_{|\alpha|}+\|B_j^{II,II}(V^\e)-B_j^{II,II}(V_\e) \|_{|\alpha|} \Big{)}\\
\leq& C(1+\|E\|_s^s)\big{(}\|E\|_{|\alpha|}^2+\|E\|_{|\alpha|}\dfrac{\|E^{II}\|_{|\alpha|}}{\e}\big{)}.
\end{split}
\end{equation}

Similarly, we deduce that
\begin{equation}\label{85}
\begin{split}
& \quad  \int_\Omega E_{\alpha}^TA_0( V^{\e}) F_3^\alpha dx\\
&=\sum\limits_{j=1}^d\int_\Omega E_{\alpha}^TA_0( V^{\e})
\Big{(}[A_j(V^{\e}),\partial_x^\alpha] \partial_{x_j}E+\dfrac{1}{\e}[B_j(V^{\e}),\partial_x^\alpha] \partial_{x_j}E\Big{)}dx\\
&\leq C\|E_\alpha\|\sum\limits_{j=1}^d\|[A_j(V^{\e}),\partial_x^\alpha] \partial_{x_j}E\|\\
&\quad+C\|E_\alpha^I\|\sum\limits_{j=1}^d\Big{\|}\dfrac{1}{\e}[B_j^{I,I}(V^\e),\partial_x^\alpha]\partial_{x_j}E^I+[B_j^{I,II}(V^\e),\partial_x^\alpha]\dfrac{\partial_{x_j} E^{II}}{\e}\Big{\|}\\
&\quad+C\dfrac{\|E_\alpha^{II}\|}{\e}\sum\limits_{j=1}^d\Big{\|}[B_j^{II,I}(V^\e),\partial_x^\alpha]\partial_{x_j}E^I+[B_j^{II,II}(V^\e),\partial_x^\alpha]\partial_{x_j}E^{II}\Big{\|}\\
&\leq C\|E_\alpha\|\sum\limits_{j=1}^d\Big{[}\big{(}\|A_j(V^\e)\|_s+\dfrac{1}{\e}\|B_j^{I,I}(V^\e)\|_s\big{)}\|E\|_{|\alpha|}\\
&\quad +\big{(}\|B_j^{I,II}(V^\e)\|_s+\|B_j^{II,I}(V^\e)\|_s+\|B_j^{II,II}(V^\e)\|_s\big{)}\dfrac{\|E^{II}\|_{|\alpha|}}{\e}\Big{]}\\
&\leq C(1+\|E\|_s^s)\Big{(}(1+\dfrac{1}{\e}\|E^{II}\|_s)\|E\|_{|\alpha|}^2+\|E\|_{|\alpha|}\dfrac{\|E^{II}\|_{|\alpha|}}{\e}\Big{)}.
\end{split}
\end{equation}
Here in the last step we have used that
$$
\|B_j^{I,I}(V^\e)\|_{|\alpha|}\leq C(1+\|E\|_s^s)\Big{(}\|E^{II}\|_{|\alpha|}+\|V^{II}_\epsilon\|_{|\alpha|}\Big{)},
$$
which follows from Proposition \ref{prop e} that $B_j^{I,I}(V)$ vanishes at equilibrium. In the same way, we have
\begin{equation}\label{86}
\begin{split}
&\quad \int_\Omega E_{\alpha}^TA_0( V^{\e}) F_4^\alpha dx\\
&=\dfrac{1}{\e^2}\int_\Omega E_{\alpha}^TA_0( V^{\e})\left( \begin{array}{c}
0 \\{}
[H(V^\e; \e), \partial_x^\alpha]E^{II} +\partial_x^\alpha\Big{[}\big{(}H(V_\e; \e)-H(V^\e; \e)\big{)} V^{II}_\e\Big{]}\\
\end{array} \right)dx\\
&\leq C\dfrac{1}{\e^2}\int_\Omega \left|E_{\alpha}^{II}\right|\left|[H(V^\e; \e), \partial_x^\alpha]E^{II} +\partial_x^\alpha\big{[}(H(V_\e, \e)-H(V^\e; \e))V^{II}_\e\big{]}\right|dx\\
&\leq C\|H(V^\e; \e)\|_s\dfrac{\|E_\alpha^{II}\|}{\e}\dfrac{\|E^{II}\|_{|\alpha|-1}}{\e}+C(1+\|E\|_s^s)\|E\|_{|\alpha|}\dfrac{\|E^{II}_\alpha\|}{\e}\dfrac{\|V^{II}_\e\|_s}{\e}\\
&\leq C(1+\|E\|_s^s)\big{(} \|E\|_{|\alpha|}+ \dfrac{\|E^{II}\|_{|\alpha|-1}}{\e} \big{)}\dfrac{\|E_\alpha^{II}\|}{\e}
\end{split}
\end{equation}
and
\begin{equation}\label{87}
\begin{split}
&\quad \int_\Omega E_{\alpha}^T\Big{(}\sum\limits_{j=1}^d \partial_{x_j}(A_0( V^{\e})A_j( V^{\e}))+\dfrac{1}{\e}\sum\limits_{j=1}^d \partial_{x_j}(A_0( V^{\e})B_j( V^{\e}))\Big{)} E_{\alpha}dx\\
&\leq C\sum\limits_{j=1}^d \Big{(}|\partial_{x_j}V^\e|_\infty+\dfrac{1}{\e}|\partial_{x_j}\big{(}A_0^{I,I}(V^\e)B_j^{I,I}(V^\e)\big{)}|_\infty\Big{)}\|E\|_{|\alpha|}^2 \\
& +C\sum\limits_{j=1}^d|\partial_{x_j}V^\e|_\infty\|E\|_{|\alpha|}\dfrac{\|E^{II}\|_{|\alpha|}}{\e}\\
&\leq C(1+\|V^\e\|_s)(1+\dfrac{\|V^{II}\|_s}{\e})\|E\|_{|\alpha|}^2+C\|V^\e\|_s\|E\|_{|\alpha|}\dfrac{\|E^{II}\|_{|\alpha|}}{\e}\\
&\leq C(1+\|E\|_s)(1+\dfrac{\|E^{II}\|_s}{\e})\|E\|_{|\alpha|}^2+C(1+\|E\|_s)\|E\|_{|\alpha|}\dfrac{\|E^{II}\|_{|\alpha|}}{\e}.
\end{split}
\end{equation}

To estimate the remaining term, we need Proposition \ref{prop e} and the equations in \eqref{71}

\begin{equation}\label{88}
\begin{split}
&\int_\Omega E_{\alpha}^T\partial_tA_0( V^{\e})E_{\alpha}dx\\
=&\int_\Omega (E_{\alpha}^{I})^T\partial_tA_0^{I,I}( V^{\e})E_{\alpha}^{I}dx+\int_\Omega (E_{\alpha}^{II})^T\partial_tA_0^{II,II}( V^{\e})E_{\alpha}^{II}dx\\
\leq & C\int_\Omega \Big{(}|E_{\alpha}^{I}|^2(|\partial_tV^{I\e}|+|V^{II\e}\partial_tV^{II\e}|+|V^{II\e}|^2|\partial_tV^{\e}|)+ |E_{\alpha}^{II}|^2|\partial_tV^{\e}|\Big{)}dx\\
\leq & C\Big{(}|\partial_tV^{I\e}|_\infty+|V^{II\e}|_\infty |\partial_tV^{II\e}|_\infty+|V^{II\e}|_\infty^2|\partial_tV^{\e}|_\infty\Big{)}\|E\|_{|\alpha|}^2+C|\partial_tV^\e|_\infty\|E^{II}\|_{|\alpha|}^2\\
\leq & C\|E\|_{|\alpha|}^2\sum\limits_{j=1}^d\Big{(}|\partial_{x_j}V^\e|_\infty+\dfrac{1}{\e}|B_j^{I,I}(V^\e)|_\infty |\partial_{x_j}V^{I\e}|_\infty+\dfrac{1}{\e}|\partial_{x_j}V^{II\e}|_\infty\\
& +|V^{II\e}|_\infty(\dfrac{1}{\e}|\partial_{x_j}V^\e|_\infty+\dfrac{1}{\e^2}|V^{II\e}|_\infty)+|V^{II\e}|_\infty^2(\dfrac{1}{\e}|\partial_{x_j}V^\e|_\infty+\dfrac{1}{\e^2}|V^{II\e}|_\infty)\Big{)}\\
&+C\|E^{II}\|_{|\alpha|}^2\sum\limits_{j=1}^d(\dfrac{1}{\e}|\partial_{x_j}V^\e|_\infty+\dfrac{1}{\e^2}|V^{II\e}|_\infty)\\
\leq &  C\|E\|_{|\alpha|}^2\left(\|V^\e\|_s+\dfrac{\|V^{II\e}\|_s}{\e}+\|V^\e\|_s\dfrac{\|V^{II\e}\|_s}{\e}+\dfrac{\|V^{II\e}\|_s^2}{\e^2}
\right)\\
&+C(\|V^\e\|_s+\dfrac{\|V^{II\e}\|_s}{\e})\|E\|_{|\alpha|}\dfrac{\|E^{II}\|_{|\alpha|}}{\e}\\
\leq & C\|E\|_{|\alpha|}\Big{(}(1+\|V^\e\|_s)(1+\dfrac{\|V^{II\e}\|_s}{\e})^2\|E\|_{|\alpha|}+(\|V^\e\|_s+\dfrac{\|V^{II\e}\|_s}{\e})\dfrac{\|E^{II}\|_{|\alpha|}}{\e}\Big{)}\\
\leq & C\|E\|_{|\alpha|}\Big{(}(1+\|E\|_s)(1+\dfrac{\|E^{II}\|_s}{\e})^2\|E\|_{|\alpha|}+(1+\|E\|_s+\dfrac{\|E^{II}\|_s}{\e})\dfrac{\|E^{II}\|_{|\alpha|}}{\e}\Big{)}.\\
\end{split}
\end{equation}

Substituting the estimates in \eqref{83}-\eqref{88} into \eqref{82}, we obtain
\begin{equation}\label{89}
\begin{split}
& \dfrac{d}{d t} \int_\Omega E_{\alpha}^TA_0(V^\e) E_{\alpha}dx\\
\leq & -2c_0\dfrac{\|E_\alpha^{II}\|^2}{\e^2}+C\e^4+C(1+\|E\|_s^s)\big{(}\|E\|_{|\alpha|}^2+\|E\|_{|\alpha|}\dfrac{\|E^{II}\|_{|\alpha|}}{\e}\big{)}\\
&+C(1+\|E\|_s^s)(1+\dfrac{1}{\e}\|E^{II}\|_s)\|E\|_{|\alpha|}^2\\
&+C(1+\|E\|_s^s)\dfrac{\|E^{II}\|_{|\alpha|-1}}{\e} \dfrac{\|E_\alpha^{II}\|}{\e}\\
&+C(1+\|E\|_s)(1+\dfrac{\|E^{II}\|_s}{\e})^2\|E\|_{|\alpha|}^2\\
&+C(1+\|E\|_s+\dfrac{\|E^{II}\|_s}{\e})\|E\|_{|\alpha|}\dfrac{\|E^{II}\|_{|\alpha|}}{\e}.
\end{split}
\end{equation}
Set
$$
D=D(t)=\dfrac{\|E\|_s}{\e}.
$$
By using the Cauchy inequality and the inequality
$$
\forall k,l>0, \ k<l,\qquad x^k\leq 1+x^l \quad\text{for all} \quad x>0,
$$
we deduce from \eqref{89} that
\begin{equation}\label{810}
\begin{split}
&\dfrac{d}{d t} \int_\Omega E_{\alpha}^TA_0(V^\e) E_{\alpha}dx+2c_0\dfrac{\|E_\alpha^{II}\|^2}{\e^2}\\
\leq & C\e^4+ C(1+D^{2s+2})\|E\|_{|\alpha|}^2+\delta\dfrac{\|E^{II}\|_{|\alpha|}^2}{\e^2}+C\dfrac{\|E^{II}\|_{|\alpha|-1}^2}{\e^2}
\end{split}
\end{equation}
with $\delta$ a parameter to be fixed.

Since $A_0(V^\e)$ is uniformly positive definite and the initial error satisfies $\|V^\e(\cdot, 0) -V_\e(\cdot, 0)\|_s=O(\e^2)$, we integrate the last inequality from $0$ to $\hat t$ with $\hat t <\min\{t_\e,t_*\}$ to obtain
\begin{equation*}
\begin{split}
C^{-1}\|E_\alpha\|^2+2c_0\int_0^{\hat t} \dfrac{\|E_\alpha^{II}\|^2}{\e^2}dt\leq & C(t_*)\e^4+ C\int_0^{\hat t}(1+D^{2s+2})\|E\|_{|\alpha|}^2dt\\
&+\delta\int_0^{\hat t}\dfrac{\|E^{II}\|_{|\alpha|}^2}{\e^2}dt+C\int_0^{\hat t}\dfrac{\|E^{II}\|_{|\alpha|-1}^2}{\e^2}dt.
\end{split}
\end{equation*}
Summing up the above inequality over the multi-index $\alpha$ with $|\alpha|\leq k$ for $k\leq s$ and taking a sufficiently small $\delta$ give
\begin{equation}\label{811}
\begin{split}
&C^{-1}\| E\|_k^2+\dfrac{c_0}{\e^2}\int_0^{\hat t}\| E^{II}\|_k^2dt\\
 \leq &C\int_0^{\hat t}\| E\|_{k}^2(1+D^{2s+2})dt+ C\dfrac{1}{\e^2}\int_0^{\hat t}\| E^{II}\|^2_{k-1}dt+C(t_*)\e^4.
\end{split}
\end{equation}
Recall that $\|\cdot\|_{-1}=0$. A simple induction based on \eqref{811} yields
\begin{equation*}
\dfrac{1}{\e^2}\int_0^{\hat t}\| E^{II}\|_k^2dt\leq C\int_0^{\hat t}\| E\|_{k}^2(1+D^{2s+2})dt+C(t_*)\e^4.
\end{equation*}
Combining this and \eqref{811} with $k=s$, we get
\begin{equation*}
 \| E\|_s^2\leq C\int_0^{\hat t}\| E\|_{s}^2(1+D^{2s+2})dt+C(t_*)\e^4.
\end{equation*}

Now we apply Gronwall's lemma to the last inequality to get
\begin{equation}\label{e-8}
 \| E\|_s^2\leq C(t_*)\e^4\exp \Big(C\int_0^{\hat t}(1+D^{2s+2})dt\Big).
\end{equation}
Since $\| E(t)\|_s=\e D(t)$, it follows from \eqref{e-8} that
\begin{equation}\label{e-9}
D({\hat t})^2\leq C(t_*)\e^2\exp \Big(C\int_0^{\hat t}(1+D^{2s+2})dt\Big)\equiv \Phi({\hat t}).
\end{equation}
Then we have
$$
\Phi'({\hat t})=C\big{(}1+D({\hat t})^{2s+2}\big{)}\Phi({\hat t})\leq C\big{(}\Phi({\hat t})+\Phi({\hat t})^{s+2}\big{)}.
$$
Applying the nonlinear Gronwall-type inequality \cite{Yong_JDE} to the last inequality gives
$$
\sup_{t\in [0,{\hat t}]}\Phi(t)\leq C(t_*)
$$
if $\Phi(0)=C(t_*)\e^2$ is sufficiently small. By using this boundedness of $\Phi(t)$ and thereby $D(t)$ due to \eqref{e-9}, we see from \eqref{e-8} that
$$
\| E(t)\|_s^2\leq C(t_*)\e^4, \quad \forall t\in [0,\min\{t_\e,t_*\}).
$$
This, together with the convergence-stability principle (Lemma \ref{61}), completes the proof of \thmref{CompTheo}.

\section*{Appendix I}
In this appendix, we give a typical choice of the generalized entropy function $s$ and dissipation matrix $\mathbf{M}$. They are
\begin{equation*}\label{Example}
s=s(\nu, u, \bb w, \bb c)=s^{eq}(\nu, u)-\frac{1}{2\alpha_1(\nu,u)}|\bb w|^2-\frac{1}{2\alpha_2(\nu,u)}|\bb c|^2
\end{equation*}
and
\begin{equation*}\label{Example_M}
\mathbf{M}\cdot  \left(\begin{array}{c}
    \bb q\\
    \theta^{-1}\bbtau
  \end{array}\right)=\left(\begin{array}{c}
    \frac{\bb q}{\lambda\theta^2}\\[2 ex]
\frac{1}{\xi}\left(\bbtau-\frac{1}{d} \hbox{tr}(\bbtau)\bb I_d\right)+\frac{1}{d\kappa} \hbox{tr}(\bbtau)\bb I_d
  \end{array}\right).
\end{equation*}
Here $\alpha_1$ and $\alpha_2$ are positive functions such that $s(\nu, u, \bb w,\bb c)$ satisfies the concavity and monotonicity in Postulate II. The simplest choice is that both $\alpha_1$ and $\alpha_2$ are constant. $\lambda, \xi$ and $\kappa$ are the usual transport coefficients for heat conduction, shear viscosity, and bulk viscosity, respectively. With such a choice of the entropy function, we have

\[
\bb q :=s_{\bb w} =-\frac{\bb w}{\alpha_1}, \quad  \bbtau :=\theta s_{\bb c}=-\frac{\theta\bb c}{\alpha_2}.
\]

If we decompose the stress tensor $\bbtau$ as $\bbtau=\bbtau^s+\frac{1}{d}\tau^b\bb I_d$ with $\bbtau^{s}=\bbtau-\frac{1}{d}\hbox{tr}(\bbtau)\bb I_d$ and $\tau^{b}=\hbox{tr}(\bbtau)$ being the respective shear and bulk stresses,
then the constitutive equations \eqref{33} can be rewritten as
\begin{eqnarray*}\label{DLE2}
\begin{array}{l}
\partial_t (\rho\alpha_1\bb q) + \nabla\cdot(\rho\alpha_1\bb v\otimes \bb q )-\nabla \theta^{-1}=-\frac{\bb q}{\lambda \theta^2},\\[2 ex]
\partial_t(\rho \theta^{-1}\alpha_2\bbtau^s) + \nabla\cdot(\bb \rho \theta^{-1}\alpha_2 \bb v\otimes \bbtau^s )+[\frac{1}{2}(\nabla \bb v+\nabla \bb v^T)-\frac{1}{d}\nabla \cdot \bb v \bb I_d]=-\frac{1}{\xi}\bbtau^s,\\[2 ex]
\partial_t(\rho \theta^{-1}\alpha_2\tau^b) + \nabla\cdot(\bb \rho \theta^{-1}\alpha_2 \bb v\tau^b )+\nabla \cdot \bb v =-\frac{1}{\kappa}\tau^b.
\end{array}
\end{eqnarray*}
These are generalizations of Cattaneo's law \cite{Ca} for heat conduction and Maxwell's law \cite{Ma} for viscoelastic fluids. They give a reasonable description of non-isothermal compressible viscoelastic fluid flows and thus generalize the isothermal model \cite{Yong_ARMA_14}.

\section*{Appendix II}
This appendix is devoted to a proof of Proposition \ref{prop51}.  We start with the following two useful lemmas.

\begin{lem}\label{lem-a1}
$f(\nu,z)$ is concave for  $(\nu,z)\in (0,+\infty)\times\mathbb{R}^n$ if and only if $g(\rho,Z)=\rho f(1/\rho,Z/\rho)$ is concave for  $(\rho,Z)\in(0,+\infty)\times\mathbb{R}^n$ where $Z=\rho z$.
\end{lem}
\begin{proof} We first show the necessity. For any $t\in[0,1]$ and  $(\rho_1,Z_1),(\rho_2,Z_2)\in(0,+\infty)\times\mathbb{R}^n$, we have
  \begin{eqnarray*}
  &&g(t\rho_1+(1-t)\rho_2,t Z_1+(1-t) Z_2)\\
  =&&(t\rho_1+(1-t)\rho_2)f(\frac{1}{t\rho_1+(1-t)\rho_2},\frac{t Z_1+(1-t)Z_2}{t\rho_1+(1-t)\rho_2})\\
  =&&(t\rho_1+(1-t)\rho_2)f(\xi\frac{1}{\rho_1}+(1-\xi)\frac{1}{\rho_2},\xi\frac{Z_1}{\rho_1}+(1-\xi)\frac{Z_2}{\rho_2} ), \ \ \xi=\frac{t\rho_1}{t\rho_1+(1-t)\rho_2},\\
  \ge&&(t\rho_1+(1-t)\rho_2)\left[\xi f(\frac{1}{\rho_1},\frac{Z_1}{\rho_1})
  +(1-\xi)f(\frac{1}{\rho_2},\frac{Z_2}{\rho_2})\right]\\
  =&& t\rho_1 f(\frac{1}{\rho_1},\frac{Z_1}{\rho_1})+(1-t)\rho_2f(\frac{1}{\rho_2},\frac{Z_2}{\rho_2})\\
  =&& tg(\rho_1,Z_1)+(1-t)g(\rho_2,Z_2).
\end{eqnarray*}
Note that $f(\nu, z)=\nu g(\frac{1}{\nu},\frac{z}{\nu})$. The same argument leads to the sufficiency.
 \end{proof}

\begin{lem}\label{lem-a2}
If $f(x,y)$ is strictly concave with  $f_y\geq0$ and $g(z)$ is concave, then $h(x,z):=f(x,g(z))$ is strictly concave.
\end{lem}
\begin{proof}
For any $t\in(0,1)$, we deduce that
\begin{eqnarray*}
&&h(tx_1+(1-t)x_2,t z_1+(1-t)z_2)\\
=&&f(tx_1+(1-t)x_2,g(tz_1+(1-t)z_2))\\
\ge && f(tx_1+(1-t)x_2,tg(z_1)+(1-t)g(z_2))\\
>&&tf(x_1,g(z_1))+(1-t)f(x_2,g(z_2))\\
=&&th(x_1,z_1)+(1-t)h(x_2,z_2).
\end{eqnarray*}
Here the first inequality uses the concavity of $g(z)$ and the monotonicity of $f(x,y)$ in $y$ simultaneously.
\end{proof}

Now we turn to prove Proposition \ref{prop51}. Recall from Postulate II that $s_u>0$ and notice that $u(\bb v, e)=e-\frac{1}{2}|\bb v|^2$ is concave. It follows from Lemma \ref{lem-a2} immediately  that $s(\nu, u(\bb v, e), \bb w, \bb c)$ is strictly concave with respect to the new variable $(\nu, \bb v, e, \bb w, \bb c)$. Moreover, we see the strict convexity of $\eta=\eta(U)$ by using Lemma \ref{lem-a1} and the definition \eqref{51}: $\eta(\rho, \rho \bb v, \rho e, \rho \bb w, \rho \bb c)=-\rho s(\nu, u(\bb v, e), \bb w, \bb c)$.

Next, we calculate the gradient
\begin{align*}\eta_U=& - \rho \left(-\frac{1}{\rho}\eta_\rho,-\frac{\rho \bb v}{\rho^2}s_u,\frac{1}{\rho}s_u, \frac{1}{\rho}s_\bb w, \frac{1}{\rho}s_\bb c\right)\\
=&(\eta_\rho,s_u\bb v, -s_u,-s_\bb w,-s_\bb c)\\
=&(\eta_\rho,\theta^{-1}\bb v, -\theta^{-1},-\bb q,- \theta^{-1}\bbtau).
\end{align*}
In this calculation, we have used that the Jacobian of the transform from $(\nu, u, \bb w, \bb c)$ to $(\rho, \rho \bb v, \rho e, \rho \bb w,\rho \bb c)$ is
\begin{equation*}
\frac{\partial(\nu, u, \bb w,\bb c)}{\partial (\rho, \rho \bb v, \rho e, \rho \bb w,\rho \bb c)}
=\begin{pmatrix}
-\frac{1}{\rho^2} & 0 & 0 & 0 & 0 \\
-\frac{\rho e}{\rho^2}+\frac{|\rho \bb v|^2}{\rho^3} & -\frac{\rho \bb v}{\rho^2} & \frac{1}{\rho} & 0 & 0 \\
 -\frac{\rho \bb w}{\rho^2} & 0 & 0& \frac{1}{\rho} & 0 \\
 -\frac{\rho \bb c}{\rho^2} & 0& 0 & 0 & \frac{1}{\rho}
 \end{pmatrix}.
\end{equation*}
Moreover, we use the Jacobian above to compute
\begin{align*}
\eta_\rho=&-s - \rho \left(s_\nu (-\frac{1}{\rho^2}) +(-\frac{1}{\rho^2})( (\rho e-\frac{|\rho \bb v|^2}{\rho }) s_u+\rho \bb w\cdot s_\bb w +\rho \bb c:s_\bb c)\right)\\
=&\frac{1}{\rho}\left(\eta+\theta^{-1}\pi-\rho \bb v\cdot \eta_{\rho \bb v}-\rho e \eta_{\rho e}-\rho \bb w\cdot \eta_{\rho \bb w}-\rho \bb c:\eta_{\rho \bb c}\right).
\end{align*}
Rewriting the last equality leads to
\begin{equation*}
\theta^{-1}\pi=\eta_U\cdot U-\eta.
\end{equation*}
This completes the proof of Proposition \ref{prop51}.

\vskip 6mm
{\bf Acknowledegment.} This work was partially supported by NSFC under grant nos. 11471185 and 11871299.

\end{document}